\documentclass[10pt]{article}

\usepackage{latexsym}
\usepackage{amsmath}
\usepackage{amsthm}
\usepackage{amssymb}
\usepackage{enumerate}
\usepackage{url}

\title{Reconstructing the topology on monoids and polymorphism clones of reducts of the rationals}

\author{John K. Truss, University of Leeds, and}
\date{Edith Vargas-Garc\'\i a, ITAM, Mexico.}

\begin{document}
\maketitle 
\newtheorem{lemma}{Lemma}[section]
\newtheorem{theorem}[lemma]{Theorem}
\newtheorem{corollary}[lemma]{Corollary}
\newtheorem{definition}[lemma]{Definition}
\newtheorem{remark}[lemma]{Remark}
\setcounter{footnote}{1}\footnotetext{Department of Pure Mathematics, University of Leeds, Leeds LS2 9JT, UK, and Departamento Acad\'emico de Matem\'aticas, ITAM.
R\'io Hondo No. 1, Col. Tizap\'an San Angel, Del. Alvaro Obreg\'on, C.P. 01080 Ciudad de M\'exico, e-mails pmtjkt@leeds.ac.uk, edith.vargas@itam.mx. The second author gratefully acknowledges the financial 
support of Asociaci\'on Mexicana de Cultura A.C.}

\newcounter{number}

\begin{abstract} We extend results from an earlier paper giving reconstruction results for the endomorphism monoid of the rational numbers under the strict and reflexive relations to the first order reducts 
of the rationals and the corresponding polymorphism clones. We also give some similar results about the coloured rationals. \end{abstract}

2010 Mathematics Subject Classification 08A35

keywords: rationals, automatic homeomorphicity, embedding, endomorphism, polymorphism clone

\section{Introduction}
In \cite{Truss1} we showed how to reconstruct the topology on the endomorphism monoid of the set of rational numbers, under 
the strict and reflexive relations $<$ and $\le$ and the polymorphism clone Pol$({\mathbb Q}, \le)$. This is a so-called 
`automatic-homeomorphicity' result, meaning that isomorphisms of a certain form must necessarily also be homeomorphisms. In 
this paper we extend these results to the reducts of the rationals described by Cameron in \cite{Cameron}, which are the 
betweenness relation, circular order, and separation relation on $\mathbb Q$. (The other reduct, namely the trivial 
structure, is already treated in \cite{Bodirsky1}.) We also consider the case of the coloured version of the rationals 
${\mathbb Q}_C$ for a set of colours with $2 \le |C| \le \aleph_0$, here just dealing with the analogues of the reducts of 
$\mathbb Q$ rather than {\em all} reducts (which according to \cite{Junker} may be quite complicated). In most cases the 
earlier results can be invoked fairly directly, or else suitably adapted. En route we need to verify the small index property 
for the automorphism group of the circular rational order $({\mathbb Q}, circ)$, and the corresponding results in the coloured 
case, which may or may not have been remarked before (or are `folklore'). 

By saying that a transformation monoid $M$ on a countable set $\Omega$ has {\em automatic homeomorphicity} is meant that any 
isomorphism from $M$ to a closed submonoid of the full transformation monoid $Tr(\Omega')$ on a countable set $\Omega'$ is 
necessarily a homeomorphism. Here the topology on $M$ is given by taking as basic open sets 
${\cal B}_{B C} = \{f \in M: fB = C\}$ where $B$ and $C$ are finite subsets of $\Omega$. The motivation for studying such 
properties is that in the case where $M$ is a (closed) subgroup of the symmetric group on $\Omega$, $M$ has automatic 
homeomorphicity if and only if it has the small index property SIP, which says that any subgroup of $M$ of index 
$< 2^{\aleph_0}$ contains the pointwise stabilizer of a finite set. Thus automatic homeomorphicity represents a statement 
which one can attempt to establish even when the statement of SIP makes no sense (in the monoid case, where we do not have 
Lagrange's Theorem).

The key steps presented in \cite{Truss1} were as follows. For $M = {\rm Emb}({\mathbb Q}, \le)$, the monoid of 
order-preserving embeddings of $\mathbb Q$, we established that any injective endomorphism of $M$ which fixes all elements of 
$G = {\rm Aut}({\mathbb Q}, \le)$, is the identity. Then invoking the fact that $G$ has the small index property 
\cite{Truss2}, that $M$ is the closure of $G$ under the above topology, and Lemma 12 from \cite{Bodirsky1}, automatic 
homeomorphicity of $M$ follows at once. To prove automatic homeomorphicity of $E = {\rm End}({\mathbb Q}, \le)$, in which $G$ 
is {\em not} dense, a more detailed analysis of the way that $E$ acts on $\Omega'$ was required; and the corresponding proof 
for the polymorphism clone Pol$({\mathbb Q}, \le)$ was given.

Here we follow a similar method for the following cases: $({\mathbb Q}, betw)$, $({\mathbb Q}, circ)$, and 
$({\mathbb Q}, sep)$, which are defined as follows:

$betw(x, y, z)$ if $x \le y \le z$ or $z \le y \le x$,

$circ(x, y, z)$ if $x \le y \le z$ or $y \le z \le x$ or $z \le x \le y$,

$sep(x, y, z, t)$ if $circ(x, y, z) \wedge circ(x, t, y)$ or  $circ(x, z, y) \wedge circ(x, y, t)$.

We also treat the analogues for these of the $C$-coloured rationals ${\mathbb Q}_C$, where $2 \le |C| \le \aleph_0$, which is 
defined to be $\mathbb Q$ together with a colouring function $F: {\mathbb Q} \to C$ such that for all $c \in C$, $F^{-1}\{c\}$ 
is a dense subset of $\mathbb Q$. This exists and is unique up to isomorphism (and is homogeneous and for the case that $C$ 
is finite, $\aleph_0$-categorical). The three structures that we treat for $\mathbb Q$ are the proper `reducts' of 
$({\mathbb Q}, \le)$, as was shown in \cite{Cameron}. The analogous structures for the coloured case are indeed reducts of 
$({\mathbb Q}_C, \le)$; however, they are not {\em all} the reducts here, and the complete list of reducts of 
$({\mathbb Q}_C, \le)$ is currently unknown; see \cite{Junker} section 6 for a discussion of this case (for finite $C$).

All these structures come in reflexive versions, as we have stated them, and strict versions. However, in the present context 
it is good enough to distinguish these by means of monoids which capture the same ideas. Thus ${\rm Emb}({\mathbb Q}, betw)$ 
for instance is the family of 1--1 maps of $\mathbb Q$ which preserve $betw$, and these are precisely the maps which preserve 
the corresponding `strict betweenness' $x < y < z \vee z < y < x$. As in \cite{Truss1}, automatic homeomorphicity is 
established for the monoid of embeddings (the endomorphisms preserving the strict relation), the monoid of all endomorphisms 
(preserving just the reflexive relation), and the polymorphism clone for the reflexive relation.

Finally we give the result from section 6 of \cite{Truss1} which deduces automatic homeomorphicity for the polymorphism clone 
generated by a monoid in the cases discussed here.

Throughout we use $G$, $M$, $E$, $P$ to stand for the automorphism group of the structure being considered, or its monoid of
embeddings, or endomorphisms, or the polymorphism clone of the reflexive relation, respectively.

We outline the methods used in \cite{Truss1}, and explain how they are adapted here. Of the three cases treated, that of 
the monoid of embeddings $M$ is the easiest, in that we can concentrate on proving a combinatorial lemma, and then appeal to
the results of \cite{Bodirsky1} mentioned to complete the proof. The combinatorial lemma proved, which is attractive in its 
own right, is that any injective endomorphism $\xi$ of $M$ which fixes $G$ pointwise is equal to the identity. That this 
statement is related to reconstruction matters is plausible, in that we are saying that somehow the monoid can be `captured' or
`described' from the group (and reconstruction from the group is just the small index property). The method for proving this
lemma presented in \cite{Truss1} involved consideration of a certain family $\Gamma$ of members of $M$, which intuitively are
those whose image is as `spread out' as possible. It was shown that members of $\Gamma$ are fixed by $\xi$, and the 
proof was completed by showing that all members of $M$ can be suitably expressed in terms of members of $\Gamma$. In fact, for 
$(\mathbb Q, \le)$ (and in this paper also for $({\mathbb Q}_C, \le)$) it was also necessary to consider variants of $\Gamma$, 
written $\Gamma^+$, $\Gamma^-$, and $\Gamma^\pm$, allowing for members of $M$ with support bounded above or below or both. Here
since we are dealing with suitable reducts, mainly based on the circular ordering of $\mathbb Q$, only the analogue of $\Gamma$
is needed. For completeness we give the construction and arguments, even though these amount to straightforward modifications 
of the ones given in \cite{Truss1}.

The harder part of the argument given in \cite{Truss1} was for the monoid of endomorphisms $E$ of $(\mathbb Q, \le)$. The 
scenario we have to consider is that in which there is an isomorphism $\theta$ from $E$ to a closed submonoid $E'$ of the full 
transformation monoid $Tr(\Omega)$ on a countable set $\Omega$, and we have to show that it is a homeomorphism. The first step
was to show that the image of $M$ under $\theta$ is a closed subset of $E'$, which was Lemma 4.1 of \cite{Truss1}. From that
we deduce by what has already been shown that on $M$, $\theta$ is a homeomorphism. Next we consider the orbits of the action
of $G$ on $\Omega$, and show by appeal to the small index property that each of these can be identified with the family
$[{\mathbb Q}]^n$ of $n$-element subsets of $\mathbb Q$ for some $n \ge 0$, termed the `rank' of the orbit. Thus we may
write $\Omega = \bigcup_{i \in I}\Omega_i$ for some index set $I$, where $\Omega_i = \{a^i_B: B \in [{\mathbb Q}]^{n_i}\}$
where $n_i$ is the rank of the orbit $\Omega_i$, and furthermore the action is `natural' in the sense that for each 
$g \in G$, $\theta(g)(a^i_B) = a^i_{gB}$. Next it is shown that this naturality extends to the action of $M$, and even to 
members of $E$ insofar as they act injectively on the relevant set $B$. For this it is fairly easy to check directly that 
$\theta$ is a homeomorphism.

Here for completeness we go over the similar arguments in reasonable detail in the first case treated, namely that of the 
betweenness relation on $\mathbb Q$. All the other cases follow a similar pattern, so in these we omit most of the details, 
concentrating on the aspects which are genuinely different.

\section{The betweenness relation on $\mathbb Q$}

We show that it is quite easy to `lift' the automatic homeomorphicity results for $({\mathbb Q}, <)$ and $({\mathbb Q}, \le)$
proved in \cite{Truss1} to the corresponding betweenness relation, essentially exploiting the fact that the original group 
has index 2 in the bigger one. The following basic lemma is needed.

\begin{lemma} \label{2.1} If $B_1$ and $B_2$ are finite subsets of $\mathbb Q$, and $G = {\rm Aut}({\mathbb Q}, \le)$, then 
$G_{B_1 \cap B_2} = \langle G_{B_1}, G_{B_2} \rangle$ (where these are the pointwise stabilizers).  \end{lemma}

\begin{proof} The fact that $\langle G_{B_1}, G_{B_2} \rangle \le G_{B_1 \cap B_2}$ is immediate, so we concentrate on the
reverse inclusion. Let $g \in G_{B_1 \cap B_2}$, and we show by induction on $n = |\{q \in B_1 \cup B_2: g(q) \neq q\}|$ that 
$g \in \langle G_{B_1}, G_{B_2} \rangle$. If $n = 0$ then already $g \in G_{B_1}$. Otherwise choose the greatest 
$q \in B_1 \cup B_2$ moved by $g$, and assume that $g(q) > q$ (which may be arranged by passing to $g^{-1}$ if 
necessary). Thus all members of $(q, \infty) \cap (B_1 \cup B_2)$ are fixed by $g$. Then there is $h \in G$ which fixes all 
members of $B_1 \cup B_2 \setminus \{q\}$ and agrees with $g$ on $q$. Thus $h \in G_{B_1}$ or $G_{B_2}$ and 
$h^{-1}g \in G_q$. But $h^{-1}g$ also fixes all members of $B_1 \cup B_2$ which are fixed by $g$, so by induction hypothesis, 
lies in $\langle G_{B_1}, G_{B_2} \rangle$. Hence also $g \in \langle G_{B_1}, G_{B_2} \rangle$. 
\end{proof}

From this lemma we can deduce that if $H$ is a subgroup of $G$ which contains $G_B$ for some finite 
$B \subseteq {\mathbb Q}$, and $B$ is of minimal size such that this is true, then $H$ equals $G_B$. To see this, take any 
$h \in H$. Then $H = hHh^{-1} \ge hG_Bh^{-1} = G_{hB}$, so by the lemma, $H \ge G_{B \cap hB}$. By minimality of $B$, 
$hB = B$, so $h \in G_{\{B\}} = G_B$. In the other cases we consider the setwise and pointwise stabilizers may not be equal, 
so there are more possibilities for $H$, but it is still always the case that if $B$ is of least size such that $H \ge G_B$, 
then $G_B \le H \le G_{\{B\}}$ (which we verify in the individual cases).

In what follows we have occasion several times to use a key result from \cite{Bodirsky1} (their `Lemma 12'), which we quote here:

\begin{lemma} \label{2.2} Let $M$ be a closed submonoid of $O^{(1)}$ whose group of invertible elements $G$ is dense in $M$ and has automatic homeomorphicity. Assume that the only
injective endomorphism of $M$ that fixes every element of $G$ is the identity function
$id_M$ on $M$. Then $M$ has automatic homeomorphicity.  \end{lemma}

\begin{theorem} \label{2.3} ${\rm Emb}({\mathbb Q}, betw$) has automatic homeomorphicity.  \end{theorem}

\begin{proof} Writing $M$ and $G$ for the monoid of embeddings of $({\mathbb Q}, betw)$ to itself, and its automorphism group, 
respectively, we first note that $G$ is dense in $M$. This is because any order-reversing member of $M$ is the composition of 
an order-preserving member of $M$ (which can be arbitrarily well approximated by members of 
$G \cap {\rm Aut}({\mathbb Q}, \le)$) and the map $i$ sending $q$ to $-q$ for all $q$ which lies in $G$. So this means that 
we can appeal to Lemma \ref{2.2} as before, and focus on consideration of an injective endomorphism $\xi$ of $M$ 
which fixes $G$ pointwise.
 
We recall the set $\Gamma$ from \cite{Truss1}, which comprises those order-preserving embeddings $f$ of $\mathbb Q$ whose image arises from a copy of the 2-coloured rationals in which each point is replaced 
by an interval isomorphic to $\mathbb Q$ and each red interval contains exactly one point of the image of $f$ and each blue interval is disjoint from the image of $f$. In the proof of Lemma 2.2 in 
\cite{Truss1} we defined $S(g) = \{(\alpha, \beta) \in G_1^2: \alpha g = g \beta\}$ for any order-preserving embedding $g$, where $G_1$ is the subgroup of order-preserving permutations. We use the same 
notation, even if $g$ is order-reversing (though the group elements considered at this point have to be order-preserving). Now let $i$ be the involution above. Then $g$ is order-preserving if and only 
if $i g$ is order-reversing, and vice versa. We now find the connection between $S(g)$ and $S(i g)$: 
$S(i g) = \{(\alpha, \beta) \in G_1^2: \alpha i g = i g \beta\} = \{(\alpha, \beta) \in G_1^2: i\alpha i g = g\beta\} = \{(i\alpha i, \beta) \in G_1^2: \alpha g = g \beta\}$ (since as 
$\alpha$ ranges over $G_1$, so does $i\alpha i$) $= \{(i\alpha i, \beta)): (\alpha, \beta) \in S(g)\}$.

It was shown in \cite{Truss1} that for any order-preserving embedding $g \in \Gamma$, and rationals $u$ and $s$,
$g(u) = s$ if and only if $\forall(\alpha, \beta) \in S(g)(\beta(u) = u \to \alpha(s) = s)$.   We can now deduce the same equivalence in the case that $g'$ is an order-reversing embedding, of the form 
$g' = i g$ for $g \in \Gamma$.
For $g'(u) = s \Leftrightarrow g(u) = i(s) \Leftrightarrow \forall(\alpha, \beta) \in S(g)(\beta(u) = u \to \alpha(i(s)) = i(s)) \Leftrightarrow \forall(\alpha, \beta) \in S(g')(\beta(u) = u \to i\alpha i(i(s)) = i(s)) \Leftrightarrow \forall(\alpha, \beta) \in S(g')(\beta(u) = u \to \alpha(s) = s)$.
The argument given towards the end of the proof of Lemma 2.2 in \cite{Truss1} shows that $S(\xi(g)) = S(g)$, and we can therefore deduce that $\xi(g) = g$ precisely as in the previous case. 
Since $\xi$ fixes all members of $\Gamma$, by our earlier result, it also fixes the order-preserving members of $M$ pointwise, and multiplying by $i$ (which is known to be fixed by $\xi$), it follows that it 
also fixes all of $M$ pointwise.

In fact it is also necessary to consider members of associated classes $\Gamma^+$, $\Gamma^-$, and $\Gamma^\pm$---see \cite{Truss1} for the precise definitions, but the details are similar and are omitted 
here.

The remaining point to check to be able to appeal to Lemma \ref{2.2}, is to verify the small index property for Aut$({\mathbb Q}, betw)$, but this follows easily from the fact that it holds for 
Aut$({\mathbb Q}, \le)$, and Aut$({\mathbb Q}, \le)$ has index 2 in Aut$({\mathbb Q}, betw)$.    \end{proof} 

Now we move on to consider the reflexive relation, in other words, the monoid $E =$ End$({\mathbb Q}, betw)$. As in 
\cite{Truss1} we are given an isomorphism $\theta$ from $E$ to a closed submonoid of $Tr(\Omega)$ for some countable set 
$\Omega$, and our task is to show that it is a homeomorphism. We may partition $\Omega$ into orbits under $G$. The 
identification of each orbit $X$ with a set of the form $[{\mathbb Q}]^n$ for some $n$ is however not quite so 
straightforward, and this is because although by the small index property, the stabilizer $G_x$ of a member $x$ of $X$ 
{\em contains} the pointwise stabilizer $G_B$ of a finite $B \subseteq {\mathbb Q}$ of minimal size, it may not {\em equal} 
it, since it may contain only order-preserving permutations, or alternatively order-preserving and reversing ones. The upshot 
of this is that $X$ may either be identified with some $[{\mathbb Q}]^n$ or the set of increasing or decreasing orderings of
$n$-element subsets of $\mathbb Q$ under the natural action.

The easiest method to reach this conclusion is to note that $G_x$ equals either the setwise or pointwise stabilizer of $B$ in 
$G$, $G_{\{B\}}$ or $G_B$. This may be deduced from the order-preserving case by considering 
$G_x \cap {\rm Aut}({\mathbb Q}, \le)$, which has index either 1 or 2 in $G_x$. In each case, the orbit $X$ containing $x$ is 
equal to $\{\theta(g)(x): g \in G\}$, but the indexing to identify it with $[{\mathbb Q}]^n$ or the set of increasing or 
decreasing orderings of $n$-element subsets of $\mathbb Q$ is a little different in the two cases.

We first treat the case that $G_x = G_{\{B\}}$. For each $g \in G$, we write $a_{gB} = \theta(g)(x)$. To justify this 
notation we note that for all $g_1, g_2 \in G$, 
$g_1B = g_2B \Leftrightarrow g_2^{-1}g_1 \in G_{\{B\}} \Leftrightarrow g_2^{-1}g_1 \in G_x \Leftrightarrow \theta(g_1)(x) = \theta(g_2)(x)$. This shows that the orbit $X$ can be precisely identified with $\{a_{gB}: g \in G\}$, and as $G$ acts 
transitively on the family of $n$-element subsets of $\mathbb Q$, also with $[{\mathbb Q}]^n$. More to the point, this 
identification also corresponds to the natural (left) action of $G$, since 
$\theta(f)(a_{gB}) = \theta(f)\theta(g)(x) = \theta(fg)(x) = a_{fgB}$.

In the second case, $G_x = G_B$. This time we write $a_{g{\underline B}} = \theta(g)(x)$ for each $g$, where $\underline B$
is $B$ together with its ordering (as a subset of $\mathbb Q$). Note that $g{\underline B}$ is ordered in the `correct' way 
(as a subset of $\mathbb Q$) if $g \in {\rm Aut}({\mathbb Q}, \le)$, and with the reverse ordering otherwise. This time the 
calculation justifying the notation is as follows:

$g_1{\underline B} = g_2{\underline B} \Leftrightarrow g_2^{-1}g_1 \in G_B \Leftrightarrow g_2^{-1}g_1 \in G_x \Leftrightarrow \theta(g_1)(x) = \theta(g_2)(x)$.

So there are two orbits of $X$ under the action of Aut$({\mathbb Q}, \le)$, namely those in which the ordering of 
$g{\underline B}$ agrees with that of $\mathbb Q$, or disagrees with it. Once more, the notation is respected by the left 
action of $G$ since $\theta(f)(a_{g{\underline B}}) = \theta(f)\theta(g)(x) = \theta(fg)(x) = a_{fg{\underline B}}$.

In the special case $|B| = 0$ or 1, $G_B = G_{\{B\}}$, and we choose the `first' case. In each case we say that $X$ 
is an orbit of `rank' $n$. 

This discussion enables us to write $\Omega = \bigcup_{i \in I}\Omega_i$ where $\Omega_i$ are the orbits, containing $x_i$ 
say, and associated with each $i \in I$ we have the rank $n_i$ of $\Omega_i$, and $t_i = 0$ or 1, its `type', being 0 if 
$G_{x_i} = G_{\{B\}}$ for some $B$, and 1 if $G_{x_i} = G_{B}$ for some $B$. We also write 
$\Omega_i = \{a^i_B: B \in [{\mathbb Q}]^{n_i}\}$ in the first case, and 
$\Omega_i = \{a^i_{\underline B}: B \in [{\mathbb Q}]^{n_i}, {\underline B} \mbox{ the increasing or decreasing ordering of } B\}$ in the second.

The first step in establishing automatic homeomorphicity is to use the result already proved for $M$. For that it is 
necessary to know that the image of $M$ under $\theta$ is a {\em closed} submonoid of $Tr(\Omega)$. This is provided by the 
following key result from \cite{Truss1} (Lemma 4.1):

If $\theta: E \to E'$ is an isomorphism where $E'$ is a closed submonoid of the full transformation monoid $Tr(\Omega)$ on a 
countable set $\Omega$, then the image $M'$ of $M$ is closed in $Tr(\Omega)$ and the restriction 
$\theta\restriction_M: M\to M'$ is a homeomorphism.

We do not need to reprove this here, since $G$ also has automatic homeomorphicity and is dense in $M$. Applying this here, we 
may appeal to automatic homeomorphicity of $M$ to deduce that on $M$, $\theta$ is a homeomorphism. 

\begin{lemma} \label{2.4} For each $f \in M$ and $i \in I$, if $t_i = 0$ then for each $B \in [{\mathbb Q}]^{n_i}$,
$\theta(f)(a^i_B) = a^i_{fB}$, and if $t_i = 1$ then for each $B \in [{\mathbb Q}]^{n_i}$ and increasing or decreasing 
ordering $\underline B$ of $B$, $\theta(f)(a^i_{\underline B}) = a^i_{f{\underline B}}$.   \end{lemma}

\begin{proof} This is proved by a continuity argument. Since $G$ is dense in $M$, there is a sequence $(g_n)$ of members of 
$G$ having $f$ as limit. We treat just the case $t_i = 0$ (and $t_i = 1$ is done similarly). Let 
$B = \{q_1, \ldots, q_{n_i}\}$ and $r_k = f(q_k)$. Then for each $k$ $f$ lies in the basic open set 
$\mathcal{B}_{qr} = \{h \in M: h(q) = r\}$, and as $g_n \to f$ there is $N_k$ such that 
$(\forall n \ge N_k)g_n \in \mathcal{B}_{q_kr_k}$. Hence if $n \ge \max_{1 \le k \le n_i}N_k$, 
$g_n \in \bigcap_{1 \le k \le n_i}{\cal B}_{q_kr_k}$, so $g_n(B) = f(B)$. 

As remarked above, the restriction of $\theta$ to $M$ is continuous. Hence $\theta(g_n) \to \theta(f)$. Let 
$\theta(f)(a^i_B) = a^j_C$. Thus $\theta(f) \in \mathcal{C}_{ijBC}$. From $\theta(g_n) \to \theta(f)$ it follows that
$(\exists N)(\forall n \ge N) \theta(g_n) \in \mathcal{C}_{ijBC}$. Hence for this $N$,
$(\forall n \ge N)\theta(g_n)(a^i_B) = a^j_C$. But we know that $\theta(g_n)(a^i_B) = a^i_{g_n(B)}$ as $g_n \in G$. Hence for
such $n$, $j = i$ and $g_n(B) = C$. Taking $n \ge N, \max_{1 \le k \le n_i}N_k$, it follows that $j = i$ and 
$C = g_n(B) = f(B)$. Thus $\theta(f)(a^i_B) = a^i_{f(B)}$ as required.  \end{proof} 

We extend this even to some actions of members of $E$.

\begin{lemma}\label{2.5} If $f \in E$, $i \in I$, if $t_i = 0$, and $a^i_B \in \Omega_i$, where $|f(B)| = n_i = |B|$, then
$\theta(f)(a^i_B) = a^i_{f(B)}$, and if $t_i = 1$, and $a^i_{\underline B} \in \Omega_i$, where $\underline B$ is the increasing or decreasing ordering of $B$, and $|f(B)| = n_i = |B|$, then 
$\theta(f)(a^i_{\underline B}) = a^i_{f({\underline B})}$.
\end{lemma}

\begin{proof} It is easiest to appeal to the methods of \cite{Truss1} by composing, in the order-reversing case, with an 
order-reversing automorphism of $\mathbb Q$ which fixes $B$ (setwise). First then suppose that $f$ is order-preserving and surjective. Then there is an order-preserving $h \in E$ such that for each $q$, $fh(q) = q$, obtained by choosing 
$h(q) \in f^{-1}(q)$, and such $h$ is necessarily injective, so lies in $M$. Furthermore, if $q \in B$ we let $hf(q) = q$ 
(possible since $f$ is 1--1 on $B$). By appealing to Lemma \ref{2.4}, 
$\theta(f)(a^i_B) = \theta(f)(a^i_{hfB}) = \theta(f)\theta(h)(a^i_{fB}) = \theta(fh)(a^i_{fB}) = a^i_{fB}$, in the case 
$t_i = 0$, with a similar argument if $t_i = 1$.

Next if $f$ is order-preserving but not necessarily surjective, as shown in \cite{Truss1} Lemma 3.3, we may write 
$f = f_1f_2$ where $f_1 \in S$ and $f_2 \in M$ are order-preserving. Then $|f_1f_2B| = |B| = |f_2B|$. Hence by the 
surjective case just done, $\theta(f_1)(a^i_{f_2B}) = a^i_{f_1f_2B} = a^i_{fB}$, so by Lemma \ref{2.4} again, 
$\theta(f)(a^i_B) = \theta(f_1)\theta(f_2)(a^i_B) = \theta(f_1)(a^i_{f_2B}) = a^i_{fB}$.

Finally, suppose that $f$ is order-reversing. Then $fj$ is order-preserving where $j$ is an (order-reversing) involution 
fixing $B$ setwise, so $\theta(f)(a^i_B) = \theta(fj)\theta(j)(a^i_B) = \theta(fj)a^i_{jB} = a^i_{fj^2B} = a^i_{fB}$.
\end{proof}

If $f \in E$ `collapses' a set~$B$, then we can certainly not deduce that $\theta(f)(a^i_B) = a^j_C$ for $j = i$,
since~$\Omega_i$ and~$\Omega_j$ will have different ranks. If $B \neq \emptyset$, in \cite{Truss1} was shown that there 
is an idempotent order-preserving endomorphism $h$ whose image is $B$, from which it follows by Lemma \ref{2.5} that 
$\theta(h)(a^i_B) = a^i_B$, and this is also valid, even if the subscript is ordered. 

\begin{lemma}\label{2.6} Let $i \in I$ and $B \in [{\mathbb Q}]^{n_i}$, $B \neq \emptyset$. If $t_i = 0$ and  
$a^i_B \in \Omega_i$, then there is an idempotent order-preserving endomorphism $h\in E$ having~$B$ as image such that 
$\theta(h)(a^i_B) = a^i_B$, and if $t_i = 1$, and $a^i_{\underline B} \in \Omega_i$, where $\underline B$ is the increasing 
or decreasing ordering of $B$, then there is an idempotent order-preserving endomorphism $h\in E$ having~$B$ as image such 
that $\theta(h)(a^i_{\underline B}) = a^i_{\underline B}$.  \end{lemma}

\begin{proof} The case when $t_i = 0$ was considered in Lemma 4.4 of \cite{Truss1}. If $t_i = 1$, then similarly to Lemma 4.4 
of \cite{Truss1} we obtain an idempotent endomorphism $h \in E$ fixing all elements of $B$ and satisfying 
${\rm{im}}(h) = B$ by subdividing $\mathbb Q$ into $|B|$ pairwise disjoint intervals, each containing a single member of $B$, 
and mapping the whole of each such interval to the member of~$B$ it contains. Since $h({\underline B}) = {\underline B}$
where $B \in [\mathbb{Q}]^{n_i}$, we can apply Lemma~\ref{2.5} to get 
$\theta(h)(a^i_{\underline B}) = a^i_{h({\underline B})} = a^i_{\underline B}$.    \end{proof}

For the proof of openness in the main theorem, we still need some information about~$C$, namely that it is contained in~$f(B)$.

\begin{lemma}\label{2.7} If $f \in E$, $i, j \in I$, and $B \in [{\mathbb Q}]^{n_i}$, $C \in [{\mathbb Q}]^{n_j}$ are such 
that $\theta(f)(a^i_B)=a^j_C$ (possibly with orderings on the subscripts, depending on the values of $t_i$, $t_j$), then 
$C \subseteq fB$. \end{lemma}
\begin{proof} By Lemma \ref{2.6}, if $B \neq \emptyset$ there is an idempotent order-preserving endomorphism $h$
whose image is $B$, from which it follows by Lemma \ref{2.5} that $\theta(h)(a^i_B) = a^i_B$, and this applies here too, even 
if the subscript is ordered. From this it follows that if $f_1, f_2 \in E$ are order-preserving, and they agree on their 
actions on $B$, then $\theta(f_1)(a^i_B) = \theta(f_2)(a^i_B)$ since $f_1h = f_2h$, so that 
$\theta(f_1)(a^i_B) = \theta(f_1)(a^i_{hB}) = \theta(f_1)\theta(h)(a^i_B) = \theta(f_1h)(a^i_B) = \theta(f_2h)(a^i_B) = \theta(f_2)(a^i_B)$. This even applies if $B = \emptyset$, since then by Lemma \ref{2.5}, $\theta(f_1)(a^i_B)$ and 
$\theta(f_2)(a^i_B)$ are both equal to $a^i_B$. If now $C \not \subseteq f(B)$, choose $q \in C \setminus f(B)$ and let 
$h \in G$ be order-preserving taking $q$ to $h(q) \not \in C$ but fixing all members of $f(B)$. Then $f$ and $hf$ agree on 
$B$, so $a^j_C = \theta(f)(a^i_B) = \theta(hf)(a^i_B) = \theta(h)\theta(f)(a^i_B) = \theta(h)(a^j_C)$, contrary to 
$hC \neq C$, and giving the result.

Finally the result for order-reversing $f$ may be deduced by composing with an order-reversing member of $G_{\{B\}}$ as in 
the proof of Lemma \ref{2.5}. 

\end{proof}

Using the ideas from above, we can demonstrate automatic homeomorphicity of $E = {\rm End}(\mathbb{Q}, betw)$.

\begin{theorem}\label{2.8} $E = {\rm End}(\mathbb{Q}, betw)$ has automatic homeomorphicity. \end{theorem}
\begin{proof} The sub-basic open sets in $E$ and $E'$ are of the form ${\cal B}_{qr} = \{f \in E: f(q) = r\}$ and 
${\cal C}_{ijBC} = \{f \in E': f(a^i_B) = a^j_C\}$ or variants of these with orderings on the subscripts $B, C$, so to
establish continuity we have to show that each $\theta^{-1}({\cal C}_{ijBC})$, $\theta^{-1}({\cal C}_{ij{\underline B}C})$
etc is open in $E$. We concentrate on the case of ${\cal C}_{ijBC}$, the others being easy modifications of the same 
argument. Let $B = \{q_1, \ldots, q_m\}$ and if $f \in \theta^{-1}({\cal C}_{ijBC})$ is order-preserving let $r_k = f(q_k)$, 
so that $f \in \bigcap_{k = 1}^m{\cal B}_{q_kr_k}$. We show that this open set is contained in 
$\theta^{-1}({\cal C}_{ijBC})$, which suffices. Let $f' \in \bigcap_{k = 1}^m{\cal B}_{q_kr_k}$. Then $f$ and $f'$ agree on 
$B$. If $f'$ is also order-preserving, by the proof of Lemma \ref{2.7}, $\theta(f')(a^i_B) = \theta(f)(a^i_B) = a^j_C$, 
which gives $f' \in \theta^{-1}({\cal C}_{ijBC})$. If $f'$ is not order-preserving, then as $f$ and $f'$ agree on $B$, we 
must have $m \le 1$, giving $t_i = 0$. Let $g \in G_B$ be an involution. Then $\theta(g)(a^i_B) = a^i_{gB} = a^i_B$, and $f$ 
and $f'g$ are order-preserving, and agree on $B$, so by what we have just shown, 
$\theta(f')(a^i_B) = \theta(f'g^2)(a^i_B) = \theta(f'g)\theta(g)(a^i_B) = \theta(f'g)(a^i_B) = \theta(f)(a^i_B) = a^j_C$, 
again giving $f' \in \theta^{-1}({\cal C}_{ijBC})$. If $f$ is order-reversing, the result is established by composing with an 
order-reversing automorphism in $G_{\{B\}}$ as before.

Next we have to show that $\theta$ is open, so we consider the image $\theta({\cal B}_{qr})$ of any 
sub-basic open set and show that this is open. Let $\theta(f) \in \theta({\cal B}_{qr})$ where $f(q) = r$, and we find 
$i, j \in I$ and $B, C \subset {\mathbb Q}$ so that $\theta(f) \in {\cal C}_{ijBC} \subseteq \theta({\cal B}_{qr})$
(possibly with $B$ and/or $C$ ordered), or in one case, $B_1, B_2, C_1, C_2 \subset {\mathbb Q}$ so that 
$\theta(f) \in {\cal C}_{ijB_1C_1} \cap {\cal C}_{ijB_2C_2} \subseteq \theta({\cal B}_{qr})$. 

The first step is to show that $|{\rm im}(f)| \ge n_i > 0$ for some $i \in I$. For if not, for every 
$i \in I$, $n_i > 0 \Rightarrow |{\rm im}(f)| < n_i$. For any $a^i_B$ in $\Omega_i$, let $\theta(f)(a^i_B) = a^j_C$ (where 
the subscripts may be ordered). By Lemma \ref{2.7}, $C \subseteq f(B)$, so $n_j \le |{\rm im}(f)|$, giving $n_j = 0$ and 
$C = \emptyset$. Let $g(q) = q + 1$, so $g \in G$. Then 
$\theta(gf)(a^i_B) = \theta(g)\theta(f)(a^i_B) = \theta(g)(a^j_\emptyset) = a^j_\emptyset = \theta(f)(a^i_B)$, so that 
$\theta(gf) = \theta(f)$. By injectivity of $\theta$, $gf = f$, contradiction.

Consider the case $t_i = 1$, and since $|{\rm im}(f)| \ge n_i$ we may choose $B$ and $C$ both of size $n_i$ such that 
$f({\underline B}) = {\underline C}$ and $q \in B$. Then by Lemma \ref{2.5}, 
$\theta(f)(a^i_{\underline B}) = a^i_{f({\underline B})} = a^i_{\underline C}$, showing that 
$\theta(f) \in {\cal C}_{ii{\underline B}\,{\underline C}}$. Now consider any member of 
${\cal C}_{ii{\underline B}\,{\underline C}}$. Since we are working in $E'$ which is the image of $E$, this has the form 
$\theta(h)$ for some $h \in E$ and $\theta(h)(a^i_{\underline B}) = a^i_{\underline C}$. By Lemma \ref{2.7}, 
$C \subseteq h(B)$, so as $|B| = |C|$, $h({\underline B}) = {\underline C}$. Since $f$ maps $q$ in $\underline B$ to the 
corresponding entry of $\underline C$, it follows that $h$ does too, and hence $h(q) = r$, which shows that 
$\theta(h) \in \theta({\cal B}_{qr})$ as required.

Now look at the case $t_i = 0$, and suppose first that for some $i$, $n_i < |{\rm im} f|$, and $f$ is order-preserving. Choose 
$B, C \in [{\mathbb Q}]^{n_i + 1}$ such that $f(B) = C$ and $q \in B$. Let $B = \{q_1, q_2, \ldots, q_{n_i+ 1}\}$ in 
increasing order, and similarly for $C = \{r_1, r_2, \ldots, r_{n_i + 1}\}$, so that $f(q_k) = r_k$ for each $k$. Let 
$B_1 = \{q_1, q_2, \ldots, q_{n_i}\}$, $B_2 = \{q_2, q_3, \ldots, q_{n_i + 1}\}$, and similarly for $C_1$, $C_2$. Then 
$\theta(f) \in {\cal C}_{iiB_1C_1} \cap {\cal C}_{iiB_2C_2}$. We show that 
${\cal C}_{iiB_1C_1} \cap {\cal C}_{iiB_2C_2} \subseteq \theta({\cal B}_{qr})$. For this, take any member $\theta(h)$ of  
${\cal C}_{iiB_1C_1} \cap {\cal C}_{iiB_2C_2}$. Then by the above calculations, $h(B_1) = C_1$ and $h(B_2) = C_2$. If $h$
is order-reversing, the first equation implies that $h(q_2) = r_{n_i - 1}$ and the second that $h(q_2) = r_{n_i + 1}$. This
contradiction shows that actually $h$ is order-preserving, and as before it follows that $h(q) = r$, and hence 
$\theta(h) \in \theta({\cal B}_{qr})$. A similar proof applies if $f$ is order-reversing.

This reduces us to the case in which for every $i$, $n_i = |{\rm im}(f)|$ or 0, and we suppose $f$ order-preserving, with a 
similar argument in the order-reversing case. Then if ${\rm im}(f) = C$, for any 
$B' \in [{\mathbb Q}]^{n_i}$ on which $f$ is 1--1, $\theta(f)(a^i_{B'}) = a^i_C$, and so if $\theta(h) \in {\cal C}_{iiBC}$, 
for some such $B$, also $\theta(h)(a^i_{B'}) = a^i_C$ for every $B'$ on which $h$ is 1--1. If $h$ is order-preserving, then we
argue as before, so suppose that $h$ is order-reversing. We show that $\theta(f) = \theta(h)$. Let 
$B = \{q_1, q_2, \ldots, q_{n_i}\}$, listed in increasing order, and let $q_k', q_k'' \in (q_k, q_{k+1})$ be such that
$f(-\infty, q_1') = h(q_{{n_i} - 1}'', \infty) = \{r_1\}$, $f(q_1', q_2') = h(q_{{n_i} - 2}'', q_{{n_1} - 1}'') = \{r_2\}$,
$\ldots, f(q_{{n_i} - 1}', \infty) = h(-\infty, q_1'') = \{r_{n_i}\}$. Let $g \in G_{\{B\}}$ be order-reversing taking each
$q_k$ to $q_{n_i + 1 - k}$ and $q_k'$ to $q_{n_i - k}''$. Then by considering the action on each interval $(-\infty, q_1')$, 
$(q_1', q_2'), \ldots, (q_{n_i - 1}', \infty)$ we see that $f = hg$, and hence 
$\theta(f)(a^i_B) = \theta(hg)(a^i_B) = \theta(h)(a^i_{gB}) = \theta(h)(a^i_B)$. If $n_i = 0$ then we get the same equality 
immediately from Lemma \ref{2.5}. It follows that $\theta(f) = \theta(h)$, but as $f$ is order-preserving and $h$ is 
order-reversing, this is contrary to injectivity of $\theta$. \end{proof}

\section{The circular ordering relation on $\mathbb Q$}

The (strict) circular order on $\mathbb Q$ is a ternary relation which may be defined by $circ(x, y, z)$ if $x < y < z$ or 
$y < z < x$ or $z < x < y$. In this section we demonstrate automatic homeomorphicity for its monoid of embeddings. We adapt 
the techniques from \cite{Truss1} section 2, already mentioned above when considering the betweenness relation. There we defined families $\Gamma$, $\Gamma^+$, $\Gamma^-$, $\Gamma^\pm$ of 
embeddings. For these we used the `2-coloured version of the rationals' ${\mathbb Q}_2$, which is taken to be the ordered 
rationals together with  colouring by 2 colours such that each colour occurs densely. Analogously we may form the {\em 
2-coloured version $C_2$} of $({\mathbb Q}, circ)$, which is taken to be $\mathbb Q$ under the same (circular) relation, 
coloured by two colours, `red' and `blue', each of which occurs densely. This again exists and is unique up to isomorphism. 
Note that for any $x, y \in {\mathbb Q}$, we may form the closed interval $[x, y] = \{z: circ(x, z, y)\}$, even if $y < x$ 
(in which case it actually equals $[x, \infty) \cup (-\infty, y]$ for `usual' intervals). 

This leads us to the analogue of the class $\Gamma$ in this case (since there are no endpoints, 
$\Gamma^+, \Gamma^-, \Gamma^\pm$ are not needed). For any embedding $f$ of $({\mathbb Q}, circ)$, we define $\sim$ by 
$x \sim y$ if $[x, y]$ or $[y, x]$ contains at most one point of the image of $f$. Each $\sim$-class is then an interval 
containing at most one point of ${\rm im} f$; if one point, then the interval is {\em red}; if no point, then it is 
{\em blue}. Then $\Gamma$ is taken to be the set of all members $f$ of $M$ all of whose $\sim$-classes are non-empty open 
intervals, and the red and blue classes form a copy of $C_2$.

For any $g \in M$ we let $\sim$ be the equivalence relation defined above, and we let $P$ be the family of all pairs $(a, b)$ 
of finite partial automorphisms of $\mathbb Q$ satisfying the following properties:

(1) $a$ is colour-preserving, and $\sim$-preserving (meaning that for 
$x, y \in {\rm dom}(a)$, $x \sim y \Leftrightarrow a(x) \sim a(y)$),

(2) if $x \in {\rm dom}(a)$ lies in a red interval containing a point $y$ of im$(g)$, then $y \in {\rm dom}(a)$,
 
(3) if $x \in {\rm im}(a)$ lies in a red interval containing a point $y$ of im$(g)$, then $y \in {\rm im}(a)$,

(4) if $x \in {\rm dom}(b)$, then $g(x) \in {\rm dom}(a)$, and $gb(x) = ag(x)$,

(5) if $x \in {\rm im}(b)$, then $g(x) \in {\rm im}(a)$, and $a^{-1}g(x) = gb^{-1}(x)$,
  
(6) if $x \in {\rm im}(g) \cap {\rm dom}(a)$, then $g^{-1}(x) \in {\rm dom}(b)$, and $gbg^{-1}(x) = a(x)$

(7) if $x \in {\rm im}(g) \cap {\rm im}(a)$, then $g^{-1}(x) \in {\rm im}(b)$, and $a^{-1}(x) \in {\rm im}(g)$. Moreover,  
$b^{-1}g^{-1}(x) = g^{-1}a^{-1}(x)$.

\begin{lemma} \label{3.1} If $f \in \Gamma$, then any $(a,b) \in P$ can be extended to a pair of automorphisms 
$(\alpha, \beta)$ of $({\mathbb Q}, circ)$ such that $\alpha f = f \beta$.   \end{lemma}

\begin{proof} Let ${\mathbb Q} = \bigcup\{A_q: q \in C_2\}$ where each $A_q$ is an open interval, circularly ordered by the 
natural relation determined from that on $C_2$, and where $|A_q \cap {\rm im} f| = 1$ if $q$ is red, and 
$A_q \cap {\rm im} f = \emptyset$ if $q$ is blue. Let ${\overline a}(q) = r$ if there is $x \in A_q \cap {\rm dom}(a)$ such 
that $a(x) \in A_r$. Then $\overline a$ is a finite colour and (circular-)order preserving partial automorphism of $C_2$. 
Extend $\overline a$ to an automorphism $\overline \alpha$ of $C_2$, and let $\alpha$ be an automorphism extending $a$ preserving im$(f)$ such that for each $q \in C_2$, $\alpha(A_q) = A_{{\overline \alpha}(q)}$. Let $\beta = f^{-1}\alpha f$.
\end{proof} 
 
\begin{lemma} \label{3.2} Any injective endomorphism $\xi$ of $M$ which fixes $G$ pointwise also fixes every member of 
$\Gamma$. \end{lemma}
\begin{proof} Let $g \in \Gamma$, and $S(g) =\{(\alpha,\beta)\in G^2: \alpha g = g\beta \}$. Consider elements $u$ and $s$ 
of ${\mathbb Q}$ with $s\neq g(u)$. We construct $(\alpha,\beta)$ such that $\alpha(s) \neq s$ and $\beta(u) = u$. We consider 
two cases:

 \begin{enumerate}
  \item If $s\in {\rm im}(g)$, then $s$  and $g(u)$ lie in different red intervals. Let $A_s$ be the red interval 
  containing $s$. Since im$(g) \cong {\mathbb Q}$, there is $t \in {\rm im}(g)$ such that $circ(g(u), s, t)$. Since $g$ is $circ$-preserving, $circ(u, g^{-1}(s), g^{-1}(t))$. Hence $a = \{(g(u), g(u)), (s, t)\}$ and 
$b = \{(u, u), ((g^{-1}(s), g^{-1}(t))\}$ are finite partial automorphisms. We can verify that $(a, b) \in P$ (as defined 
before the previous lemma). 
  \item If $s \notin {\rm{im}}(g)$, we consider two cases:
  \begin{enumerate}
  \item[(i)] If $s$ lies in a blue interval $A_q$, we choose $t \neq s$ in the same interval. Since 
  $A_q \cong {\mathbb{Q}}$, $a = \{(s, t), (g(u), g(u))\}$ and $b = \{(u, u)\}$ are finite partial automorphisms. Again 
$(a,b) \in P$.
 \item[(ii)] If $s$ lies in a red interval $A_r$, with $r \in {\rm im}(g)$, we choose $t \neq s$ in $A_r$ on the same side of $s$ (which also allows for the possibility that $r = g(u)$), meaning that $circ(g(u), r, s)$ and $circ(g(u), t, s)$, or $g(u) = r$ and $circ(g(u), t, s)$, or $circ(g(u), s, r)$ and $circ(g(u), s, t)$. Then $a = \{(g(u), g(u)),(r,r), (s, t)\}$ and 
$b = \{(u, u), (g^{-1}(r), g^{-1}(r))\}$ are finite partial automorphisms, and once more we can verify that $(a,b)\in P$.
 \end{enumerate}
  \end{enumerate}
  
In each case we can extend $(a,b)$ to $(\alpha,\beta)$ such that $\alpha g = g \beta$ by appealing to Lemma \ref{3.1}, thus
$(\alpha,\beta)$ lies in $S(g)$, and satisfies $\beta(u) = u$, $\alpha(s) = t \neq s$. Now the element $g(u)$ can be 
recovered from $S(g)$, namely as

\begin{equation}\label{Eq:Recover}
g(u)=s\iff \forall (\alpha,\beta)\in S(g)~\left(\beta(u)=u\rightarrow \alpha(s)=s\right) 
\end{equation}
For if $g(u)=s$ and $\left(\alpha,\beta\right) \in S(g)$ with $\beta(u)=u$, then 
$\alpha(s)=\alpha\left(g(u)\right)=g\beta(u)=g(u)=s.$\\
Conversely, if $g(u)\neq s$, then by the above we can construct $\left(\alpha,\beta\right)\in S(g)$ such that $\beta(u)=u$ and 
$\alpha(s)\neq s$.

Finally, from Equation (\ref{Eq:Recover}) we obtain $\xi(g)=g$,  
\begin{align*}
 (u,s)\in g &\stackrel{(\ref{Eq:Recover})}{\iff} \forall (\alpha,\beta)\in 
 S(g)~\left(\beta(u)=u\rightarrow \alpha(s)=s\right)\\
  &\iff \forall (\alpha,\beta)\in S\left(\xi \left(g\right)\right)~
  \left(\beta(u)=u\rightarrow \alpha(s)=s\right) \\
  &\iff \left(u,s\right) \in \xi(g)
\end{align*}
\end{proof}       
 
Now we consider how the members of $\Gamma$ and $M$ interact. If $g \in \Gamma$ and $f \in M$, then any $\sim_{gf}$-class is a 
union of a convex family of $\sim_g$-classes. This is because im$(gf) \subseteq {\rm im}(g)$ and so  
$x \sim_g y \Rightarrow |[x, y] \cap {\rm im}(g)| \le 1 \mbox{ or } |[y, x] \cap {\rm im}(g)| \le 1 \Rightarrow |[x, y] \cap {\rm im}(gf)| \le 1 \mbox{ or } |[y, x] \cap {\rm im}(gf)| \le 1 \Rightarrow x \sim_{gf} y$.
Since all $\sim_g$-classes are isomorphic to $\mathbb Q$, so are all the $\sim_{gf}$-classes. The family of red $\sim_{gf}$ 
classes is ordered like $\mathbb Q$, since it corresponds precisely to the image of $gf$, which is a copy of $\mathbb Q$. And 
the blue $\sim_{gf}$ classes occupy some cuts among the red ones. Two distinct blue $\sim_{gf}$ classes must occupy distinct 
cuts, as if they had no red $\sim_{gf}$ class between them, then by definition of $\sim$, they'd have to be in the same 
$\sim_{gf}$-class. This means that we may write $\mathbb Q$ as a disjoint union of sets $A_q$ for $q$ lying in some subset $Q$ 
of $C_2$, where $A_q \cong {\mathbb Q}$ and all the red members of $C_2$ lie in $Q$. This describes the general situation. 
Depending on the precise values of $g$ and $f$, we may find that $gf \in \Gamma$ or not. We first see that if they both lie 
in $\Gamma$, then the product necessarily does too. 

\begin{lemma} \label{3.3} If $g_1$ and $g_2$ lie in $\Gamma$ then so does $g_2g_1$. \end{lemma}

\begin{proof} From the above remarks, we just need to see that between any two $g_2g_1$-red intervals there is a $g_2g_1$-blue 
one, where this now means in the sense of the circular order. Let $g_2g_1x$ and $g_2g_1y$ lie in distinct intervals. Since 
$g_1 \in \Gamma$, there is a $g_1$-blue interval $(a, b) \subseteq (g_1x, g_1y)$, and its endpoints $a$ and $b$ are 
irrationals which are limits of points of im($g_1$). Let $a = \sup g_1a_n$, $b = \inf g_1b_n$ where $(a_n)$ is an increasing 
sequence, and $(b_n)$ is a decreasing sequence. Clearly $g_2(a, b)$ is disjoint from im($g_2g_1$). It 
is contained in a $g_2g_1$-blue interval (which therefore lies (strictly) in between $g_2g_1x$ and $g_2g_1y$) because the only 
way in which it could lie in a $g_2g_1$-red interval $(c, d)$ would be if there was a single point $g_2g_1(z)$ of im($g_2g_1$)
lying in it; but then $g_2g_1a_n < g_2g_1z < g_2g_1b_n$ for all $n$, giving $g_1a_n < g_1z < g_1b_n$ so $g_1z \in (a, b)$, 
contrary to $(a, b) \cap {\rm im}(g_1) = \emptyset$, so this cannot happen. (It is possible that the $g_2g_1$-blue interval is
larger than the convex hull of $g_2(a, b)$, but this does not affect the argument.)

To conclude, note that by definition of $\sim$, there cannot be consecutive blue intervals, or a consecutive pair of red/blue 
intervals, and as the red intervals are ordered like $\mathbb Q$ there are no two consecutive red intervals either. From this 
it easily follows that the family of intervals is ordered like $C_2$.  \end{proof}

\begin{lemma} \label{3.4} For any $f \in M$, there is $g \in \Gamma$ such that $gf \in \Gamma$. \end{lemma}

\begin{proof} We start by taking any $g_1 \in \Gamma$, and then we see that we can describe $g_1f$ 'fairly' well. Then we 
take another $g_2 \in \Gamma$, which will be chosen so that $g_2g_1f \in \Gamma$. Appealing to Lemma \ref{3.3}, we may let 
$g = g_2g_1$ to conclude the proof.

By the discussion above, there is a subset $Q$ of $C_2$ containing all the red points, such that 
${\mathbb Q} = \bigcup_{q \in Q}A_q$ where the $A_q$ are copies of $\mathbb Q$ such that $circ(r, s, t)$ in $Q$ implies that
the corresponding $A_r, A_s, A_t$ are circularly ordered in the same way (as sets) and if $q \in Q$ is red, then $A_q$ is a 
$g_1f$-red interval, and if it is blue, then $A_q$ is a $g_1f$-blue interval. Now we choose a countable dense set $B$ of 
('blue') irrationals such that the family of sets $A_q$ for red $q \in Q$ and the set of members of $B$ which are cuts of this 
family, together form a copy of $C_2$. Note that $B$ will have a lot more members than these particular cuts, but these are 
the crucial ones which will ensure that our $g_2g_1f$ lies in $\Gamma$. Note that in addition ${\mathbb Q} \cup B$ {\em also} 
forms a copy of $C_2$, and we use it to find $g_2 \in \Gamma$. Now each $g_1f$-red interval gives rise to a $g_2g_1f$-red 
interval. This is because it clearly still just has one point in the image, and it doesn't extend `any further' because of the 
presence of the dense set $B$. The images of the members of $B$ which were inserted densely between the sets $A_q$ for red 
$q \in Q$ are $g_2g_1f$-blue intervals which enable us to see that the result is itself a copy of $C_2$.    \end{proof}

\begin{corollary} \label{3.5} Any injective endomorphism $\xi$ of $M$ which fixes $G$ pointwise also fixes every member of 
$M$. 
\end{corollary}
\begin{proof} Let $f \in M$. By Lemma \ref{3.4}, there is $g \in \Gamma$ such that $gf \in \Gamma$. By Lemma \ref{3.2}, $\xi$ 
fixes $g$ and $gf$. Hence $\xi(g)\xi(f) = \xi(gf) = gf = \xi(g)f$. Since $g$ is left cancellable, so is $\xi(g)$, and hence 
$\xi(f) = f$.   \end{proof} 

\begin{lemma} \label{3.6} Aut$({\mathbb Q}, circ)$ has the small index property. \end{lemma}     
\begin{proof} This follows easily from the observation that ${\rm Aut}({\mathbb Q}, <)$ has countable index in 
${\rm Aut}({\mathbb Q}, circ)$ (as follows by the orbit-stabilizer theorem) and the fact that Aut$({\mathbb Q}, <)$ has
the small index property (\cite{Truss1}). \end{proof}

\begin{theorem} \label{3.7} ${\rm Emb}({\mathbb Q}, circ)$ has automatic homeomorphicity. \end{theorem}     
\begin{proof} This follows from Corollary \ref{3.5} and Lemma \ref{2.2}, since $G$ is dense in $M$, and $G$ has the 
small index property, so that by \cite{Bodirsky1} we know that $G$ has automatic homeomorphicity. \end{proof}

We now adapt the ideas of \cite{Truss1} and section 2 to demonstrate automatic homeomorphicity for 
${\rm End}({\mathbb Q}, circ)$. Once more by the small index property, if $H$ is a subgroup of $G$ of countable index, there
is a minimal finite subset $B$ of $\mathbb Q$ such that $G_B \le H$. To show that $H \le G_{\{B\}}$, we deduce from Lemma 
\ref{2.1} its analogue in the current situation.

\begin{lemma} \label{3.8} If $B_1$ and $B_2$ are finite subsets of $\mathbb Q$, and $G = {\rm Aut}({\mathbb Q}, circ)$, then 
$G_{B_1 \cap B_2} = \langle G_{B_1}, G_{B_2} \rangle$.  \end{lemma}

\begin{proof} We exploit the fact that the stabilizer $G_a$ of any point $a \in {\mathbb Q}$ is isomorphic to 
${\rm Aut}({\mathbb Q}, \le)$, and can then deduce the result from Lemma \ref{2.1}. As before we just have to check that
$G_{B_1 \cap B_2} \le \langle G_{B_1}, G_{B_2} \rangle$. Pick $a \in B_1 \cap B_2$ if this is non-empty. Then 
$g \in (G_a)_{B_1 \cap B_2} \le \langle (G_a)_{B_1}, (G_a)_{B_2} \rangle$ by Lemma \ref{2.1}, 
$\le \langle G_{B_1}, G_{B_2} \rangle$.

If $B_1 \cap B_2 = \emptyset$, we start by writing an arbitrary $f \in G$ as the product of two elements, each having a fixed 
point. Take any $a \in {\mathbb Q}$, and let $b = f(a)$. There is $h \in G$ taking $a$ to $b$, and fixing some (rational) 
point of $(b, a)$. Then $h$ and $h^{-1}f$ each has a fixed point (since $h^{-1}f$ fixes $a$). Given this observation, it 
suffices to show that any member of $g$ of $G$ having a fixed point lies in $\langle G_{B_1}, G_{B_2} \rangle$. Let $a$ be 
fixed by $g$. Running the same argument as in the first paragraph, we find that 
$g \in G_{(B_1 \cup \{a\}) \cap (B_2 \cup \{a\})} = (G_a)_{B_1 \cap B_2} \le \langle (G_a)_{B_1}, (G_a)_{B_2} \rangle \le \langle G_{B_1}, G_{B_2} \rangle$.
\end{proof}

Now that we know that $H \le G_{\{B\}}$, we need to consider what the options are for such $H$ (in section 2 there were only 
two). This time, if $B = \{b_0, b_1, \ldots , b_{n-1}\}$ in cyclic order is non-empty (that is, $n \ge 1$), 
$|G_{\{B\}}: G_{B}| = |B|$, since the cyclic ordering on $B$ has to be preserved and each cyclic permutation is possible. It 
easily follows that for some factor $m$ of $n$, if $s_m(b_i) = b_{i + m}$ where the subscripts are taken modulo $n$, then 
$H = \{g \in G: g \mbox{ acts on $B$ as a power of } s_m\}$. Let us say that $m$ is the `type' of the orbit. Given this, we 
can just adapt the machinery from section 2. Namely, $\Omega$ may be written as the union of $G$-orbits $\Omega_i$ for 
$i \in I$, $n_i$ and $m_i$ are specified, and for each $i$, $\Omega_i$ is a family of elements of the form 
$a^i_{{\underline B}_{m_i}}$ where $B \in [{\mathbb Q}]^{n_i}$. Here, ${\underline B}_m$ is the set of images of 
$\underline B$ under powers of $s_m$, so that the orderings of $B$ which arise are in the correct anticlockwise cyclic order, 
and form an orbit under $\langle s_m \rangle$. This is all done so that the action of $\theta$ is compatible with this 
enumeration for members of $G$. More precisely, we let $\Omega_i = \{a^i_{h{\underline B}_{m_i}}: h \in G\}$, where
$a^i_{h{\underline B}_{m_i}} = \theta(h)a^i_{{\underline B}_{m_i}}$. The point is that for $g, h \in G$,
$a^i_{g{\underline B}_{m_i}} = a^i_{h{\underline B}_{m_i}} \Leftrightarrow \theta(h^{-1}g)$ fixes 
$a^i_{{\underline B}_{m_i}} \Leftrightarrow$ $h^{-1}g$ acts on $\underline B$ as a power of $s_{m_i}$.

Given this background, our remaining task is to show how the machinery developed in the previous section for the betweenness 
relation carries over to this setting. The analogue of Lemma \ref{2.4} holds here by a similar continuity argument, and the 
analogues of Lemma \ref{2.5} and \ref{2.7} also carry across straightforwardly. Next we have the analogue of Lemma \ref{2.6}.

\begin{lemma} \label{3.9} Let $i \in I$ and $B \in [{\mathbb Q}]^{n_i}$ where $n_i \neq 0$, and $\Omega_i$ have type $m_i$. 
Let $a^i_{{\underline B}_{m_i}} \in \Omega_i$. Then there is an idempotent endomorphism $h \in E$ having image $B$ such that 
$\theta(h)$ fixes $a^i_{{\underline B}_{m_i}}$. \end{lemma}
\begin{proof} Subdivide the circularly ordered $\mathbb Q$ into $n_i$ pairwise disjoint intervals, each containing a single 
member of $B$. Then $h({\underline B}) = {\underline B}$, so $h$ also fixes ${\underline B}_{m_i}$, so by the analogue of 
Lemma \ref{2.5} for this case, 
$\theta(h)(a^i_{{\underline B}_{m_i}}) = a^i_{h{\underline B}_{m_i}} = a^i_{{\underline B}_{m_i}}$. \end{proof} 

The final result of this section is as follows.

\begin{theorem} \label{3.10} ${\rm End}({\mathbb Q}, circ)$ has automatic homeomorphicity. \end{theorem}
\begin{proof} Let $E'$ be a closed submonoid of $Tr(\Omega)$ where $|\Omega| = \aleph_0$, and let $\theta$ be an isomorphism 
from $E$ to $E'$, which we have to show is a homeomorphism. We decompose $\Omega$ into orbits $\Omega_i$ as above, and 
this time the sub-basic open sets in $E$ and $E'$ are of the form ${\cal B}_{qr} = \{f \in E: f(q) = r\}$ and 
${\cal C}_{ij{\underline B}_{m_i}{\underline C}_{m_j}} = \{f \in E': f(a^i_{{\underline B}_{m_i}}) = a^j_{{\underline C}_{m_j}}\}$, so for continuity we have to show that each 
$\theta^{-1}({\cal C}_{ij{\underline B}_{m_i}{\underline C}_{m_j}})$ is open in $E$, and for openness that each 
$\theta({\cal B}_{qr})$ is open in $E'$.  

For openness of $\theta^{-1}({\cal C}_{ij{\underline B}_{m_i}{\underline C}_{m_j}})$, let 
$B = \{q_1, q_2, \ldots, q_{n_i}\}$, where $B$ is listed in increasing order. Let 
$f \in \theta^{-1}({\cal C}_{ij{\underline B}_{m_i}{\underline C}_{m_j}})$ and let $f(q_k) = r_k$ (the $r_k$ need not be 
distinct). Thus $f \in \bigcap_{k = 1}^{n_i} {\cal B}_{q_kr_k}$, and we have to show that this set is contained in 
$\theta^{-1}({\cal C}_{ij{\underline B}_{m_i}{\underline C}_{m_j}})$. Let $f'$ be any member of 
$\bigcap_{k = 1}^{n_i} {\cal B}_{q_kr_k}$. Thus for each $k$, $f'(q_k) = r_k$, and hence $f$ and $f'$ agree on $B$. By Lemma 
\ref{3.9} there is an idempotent $h \in E$ with image $B$ such that $\theta(h)$ fixes $a^i_{{\underline B}_{m_i}}$. Then 
$f'h = fh$, and so 
$\theta(f')(a^i_{{\underline B}_{m_i}}) = \theta(f'h)(a^i_{{\underline B}_{m_i}}) = \theta(fh)(a^i_{{\underline B}_{m_i}}) =
\theta(f)(a^i_{{\underline B}_{m_i}}) = a^j_{{\underline C}_{m_j}}$, and therefore 
$f' \in \theta^{-1}({\cal C}_{ij{\underline B}_{m_i}{\underline C}_{m_j}})$ as required.

Next we show that $\theta({\cal B}_{qr})$ is open for any $q, r$. As in the proof of Theorem \ref{2.8} we may find $i$ such 
that $n_i \le |{\rm im}(f)|$. Choose $B$ and $C$ of size $n_i$ with $q \in B$ and such that $f(B) = C$. By the analogue of 
Lemma \ref{2.5}, $\theta(f)(a^i_{{\underline B}_{m_i}}) = a^i_{f{\underline B}_{m_i}} = a^i_{{\underline C}_{m_i}}$, showing 
that $\theta(f) \in {\cal C}_{ii{\underline B}_{m_i}{\underline C}_{m_i}}$. Let $\theta(h)$ lie in this set. As before,
$h(B) = C$, and in fact $h({\underline B}_{m_i}) = {\underline C}_{m_i}$. As before, the problem is that we do not know
that $h$ takes $q$ to $r$. For this, we follow a similar strategy to that adopted in the proof of Theorem \ref{2.8}.
First if for some $i$, $|{\rm im}(f)| > n_i$ we find `overlapping' cyclically ordered sequences of length $n_i$, and use
the extra room thus created to recover sufficiently the structure, so that endomorphisms lying in the intersection of
two sets of the form ${\cal C}_{ii{\underline B}_{m_i}{\underline C}_{m_i}}$ must take $q$ to $r$. Finally, if 
$n_i = |{\rm im}(f)|$ or 0 for every $i$, and the endomorphism $h$ which arises in the proof satisfies $h(q) \neq r$, we show 
that $\theta(h) = \theta(f)$, contrary to the injectivity of $\theta$. More precisely, let $B = \{q_1, q_2, \ldots, q_{n_i}\}$
in increasing order, and $C = \{r_1, r_2, \ldots, r_{n_i}\}$ be enumerated so that $f(q_k) = r_k$ for each $k$. Let
$h(q_k) = r_{k+t}$ for each $k$, fixed $t$, where $t \neq 0$. Then there is $g \in G$ taking $q_k$ to $q_{k - t}$ for
each $k$ (where all the suffices are taken modulo $n_i$), and we find that $hg(q_k) = h(q_{k - t}) = r_k = f(q_k)$, giving
$hg = f$. Here the fact that $h({\underline B}_{m_i}) = {\underline C}_{m_i}$ ensures that $m_i$ divides $t$ from which it follows that $g$ fixes ${\underline B}_{m_i}$. Therefore 
$\theta(f)(a^i_{{\underline B}_{m_i}}) = \theta(hg)(a^i_{{\underline B}_{m_i}}) = \theta(h)(a^i_{g{\underline B}_{m_i}})
= \theta(h)(a^i_{{\underline B}_{m_i}})$ and so $\theta(f) = \theta(h)$, as stated. \end{proof}

\section{The separation relation on $\mathbb Q$}

Since Aut$({\mathbb Q}, circ)$ has index 2 in Aut$({\mathbb Q}, sep)$, we may use the `same' method as in section 2 to deduce 
automatic homeomorphicity for Emb$({\mathbb Q}, sep)$ from the corresponding result for Emb$({\mathbb Q}, circ)$. The main 
step as usual is to consider an injective endomorphism $\xi$ of Emb$({\mathbb Q}, sep)$ to itself, which fixes all group 
elements, and show that it must be the identity. We fix some involution $i$ in Aut$({\mathbb Q}, sep)$, which interchanges the sets of 
orientation preserving and orientation reversing members of Emb$({\mathbb Q}, sep)$, and we consider the class $\Gamma$ as in section 3. If for any $g \in M$ we define
$S(g)$ to be $\{(\alpha, \beta) \in G_1^2: \alpha g = g \beta\}$, where $G_1 = {\rm Aut}({\mathbb Q}, circ)$, the same calculations used in section 3 for members of 
Emb$({\mathbb Q}, circ)$ apply to Emb$({\mathbb Q}, sep)$ to show that $\xi$ fixes all members of $\Gamma$, and hence (since $i$ is necessarily fixed by $\xi$) also all members of $M$.
This establishes the following result.

\begin{theorem} \label{4.1} ${\rm Emb}({\mathbb Q}, sep)$ has automatic homeomorphicity. \end{theorem}  

We now indicate how the methods of sections 2 and 3 are adapted in this case to yield proofs of automatic homeomorphicity
of End(${\mathbb Q}, sep$). Once more if $H$ is a subgroup of $G = {\rm Aut}({\mathbb Q}, sep)$ of countable index, then 
there is a unique (finite, minimal) $B \subset {\mathbb Q}$ such that $G_B \le H \le G_{\{B\}}$. If the two stabilizers are
equal, then as usual we can identify the orbit with $[{\mathbb Q}]^n$ for $n = |B|$. Otherwise, $|B| \ge 2$,   
$|G_{\{B\}}: G_B| = 2n$, and we have to consider rotations and reflections. If we let $B = \{b_0, b_1, \ldots, b_{n-1}\}$ in 
increasing order, then the value of $H$ can be `captured' a combination of the types from the two previous sections, that is, 
it will be a pair consisting of 1 or 0, to tell us whether $H$ has an orientation-reversing member or not, and a factor $m$ 
of $n$ such that the orientation-preserving subgroup of $H$ acts on $B$ as a power of $s_m$ defined at the end of section 3.
In section 3 we were essentially considering the action of a cyclic group, but here the corresponding action is dihedral. We 
omit the details, but state the main theorem which applies here, and which is proved by methods similar too those in sections 
2 and 3 (with some adaptations).

As usual, by the small index property, if $H$ is a subgroup of $G$ of countable index, there is a minimal finite subset $B$ of 
$\mathbb Q$ such that $G_B \le H$, and using a combination of the tricks from sections 2 and 3, also $H \le G_{\{B\}}$, so 
that $B$ is uniquely determined. This time there are extra options for what $H$ can be, obtained by reversing the orientation 
of $B$ (so it is essentially the dihedral group that is now acting). If we write 
${\underline B} = (b_0, b_1, \ldots , b_{n-1})$ in increasing order, we look at the family of sequences that arise by applying 
members of $H$. Since $H \le G_{\{B\}}$, all members of $H$ preserve the {\em set} $B$, but may perform `rotations' or
`reflections'. As in section 3, we can capture the possibilities via the set of images under $H$, which are indexed by 
subgroups of the dihedral group of order $2n$. We write the set as ${\underline B}_H$ (where it isn't really $H$ which is 
relevant---rather its induced action on $\underline B$).

Thus we may write $\Omega$ as the union of $G$-orbits $\Omega_i$ for $i \in I$, $n_i$ and $H_i$ are specified, and for each 
$i$, $\Omega_i$ is the family of elements of the form $a^i_{{\underline B}_{H_i}}$ where $B \in [{\mathbb Q}]^{n_i}$. The same 
lemmas as before are now proved in this case. We obtain the analogue of Lemma \ref{2.4} by means of a continuity argument, and
this leads to the analogue of Lemma \ref{2.5}. The analogue of Lemma \ref{2.6} is as follows.

\begin{lemma} \label{4.2} Let $i \in I$ and $B$ be a non-empty member of $[{\mathbb Q}]^{n_i}$. Then there is an
idempotent separation-preserving endomorphism $h \in E$ having $B$ as image such that 
$\theta(h)$ fixes $a^i_{{\underline B}_{H_i}}$. \end{lemma} 

For this we can take $h$ to be orientation-preserving, so use the same method as in Lemma \ref{3.9}.

Lemma \ref{2.7} carries over straightforwardly to the new situation, and the main result is as follows.

\begin{theorem} \label{4.3} ${\rm End}({\mathbb Q}, sep)$ has automatic homeomorphicity. \end{theorem}
\begin{proof} For this we use a combination of the methods of Theorems \ref{2.8} and \ref{3.10}. In fact the group is a degree 
2 extension of that for the circular ordering, so this case bears the same relationship to the circular ordering as does the 
betweenness relation to the linear ordering. 

To give a few details, once more, let $E'$ be a closed submonoid of $Tr(\Omega)$ where $|\Omega| = \aleph_0$, and let 
$\theta$ be an isomorphism from $E$ to $E'$. Write $\Omega = \bigcup_{i \in I}\Omega_i$ where 
$\Omega_i = \{a^i_{{\underline B}_{H_i}}: B \in [{\mathbb Q}]^{n_i}\}$, $H_i$ a subgroup of the dihedral group of order 
$2n_i$. The sub-basic open sets in $E$ and $E'$ are of the form ${\cal B}_{qr} = \{f \in E: f(q) = r\}$ and 
${\cal C}_{ij{\underline B}_{H_i}{\underline C}_{H_j}} = \{f \in E': f(a^i_{{\underline B}_{H_i}}) = a^j_{{\underline C}_{H_j}}\}$. The proof that $\theta$ is continuous is as before. For openness we show that each 
$\theta({\cal B}_{qr})$ is open in $E'$. We find $i$ such that $n_i \le |{\rm im}(f)|$, and choose $B$ and $C$ of size $n_i$ 
with $q \in B$ and such that $f(B) = C$. As before, $\theta(f)(a^i_{{\underline B}_{H_i}}) = a^i_{f{\underline B}_{H_i}} = a^i_{{\underline C}_{H_i}}$, showing that $\theta(f) \in {\cal C}_{ii{\underline B}_{H_i}{\underline C}_{H_i}}$. Let 
$\theta(h)$ lie in this set. Then $h(B) = C$, and $h({\underline B}_{H_i}) = {\underline C}_{H_i}$. If for some $i$, 
$|{\rm im}(f)| > n_i$ we argue as for the circular ordering case. If however $n_i = |{\rm im}(f)|$ or 0 for every $i$, and 
$h(q) \neq r$, we show that $\theta(h) = \theta(f)$, contrary to the injectivity of $\theta$. Assume that $f$ is 
orientation-preserving (with a similar argument in the orientation-reversing case). If $B = \{q_1, q_2, \ldots, q_{n_i}\}$
in increasing order, and $C = \{r_1, r_2, \ldots, r_{n_i}\}$ are enumerated so that $f(q_k) = r_k$ for each $k$, and $h$ is
also orientation-preserving, then we use the argument from Theorem \ref{3.10}, where if $h(q_k) = r_{k+t}$ for 
each $k$, we use $g \in G$ taking $q_k$ to $q_{k - t}$ for each $k$. If however $h$ is orientation-reversing, it must take 
the form $h(q_k) = r_{t-k}$ for some fixed $t$, and instead we find $g \in G$ such that $g(q_k) = q_{t - k}$ for each $k$.
This gives $hg = f$. Since $f({\underline B}_{H_i}) = {\underline C}_{H_i}$ and 
$h({\underline B}_{H_i}) = {\underline C}_{H_i}$, it follows that $hg({\underline B}_{H_i}) = h({\underline B}_{H_i})$, and 
as $h$ is 1--1 on $B$, that $g({\underline B}_{H_i}) = {\underline B}_{H_i}$. Therefore
$\theta(f)(a^i_{{\underline B}_{H_i}}) = \theta(hg)(a^i_{{\underline B}_{H_i}}) = 
\theta(h)(a^i_{g{\underline B}_{H_i}}) = \theta(h)(a^i_{{\underline B}_{H_i}})$, showing that $\theta(f) = \theta(h)$, 
contrary to $\theta$ injective. \end{proof}

\section{Automatic homeomorphicity of the polymorphism clones on $\mathbb Q$ for the reflexive case}

Our aim in this section is to carry across the results from \cite{Truss1} for the polymorphism clone of the rational numbers 
under the reflexive ordering to the reducts discussed earlier in the paper, betweenness, circular order, and separation 
relations. For definitions of the relevant notions here we refer the reader to~\cite{Bodirsky1}, but mention a few notations 
that are needed. Denoting by $\mathcal{O}_{A}$ the collection of all finitary operations $f\colon A^n\to A$ ($n\geq 0$) on a 
set~$A$, a subset $C\subseteq \mathcal{O}_{A}$ is called a (`concrete') clone on~$A$ if it is closed under the operations of 
composition when defined (that is, the `arities' are correct) and it contains all `projections'. These are the maps
$\pi_i^{(n)}\colon A^n \to A$ given by $\pi_i^{(n)}(a_1, a_2, \ldots, a_n) = a_i$, where $1 \le i \le n$. The collection of 
all polymorphisms of a relational structure always forms a clone, and clones arising in this way are precisely the ones that 
are topologically closed. Of central interest here are the clones $\mathop{\mathrm{Pol}}(\mathbb{Q},betw)$ of 
polymorphisms of $(\mathbb{Q}, betw)$ and $\mathop{\mathrm{Pol}}(\mathbb{Q},circ)$ of polymorphisms of 
$(\mathbb{Q}, circ)$, which are the families of all $n$-ary functions on~$\mathbb Q$ for $n \ge 0$ that preserve~$betw$ and 
~$circ$, respectively. Spelling out precisely what this means, $f\colon\mathbb{Q}^n \to \mathbb{Q}$ lies in 
$\mathop{\mathrm{Pol}}(\mathbb{Q},betw)$ provided that if 
$(a_1, a_2, \ldots, a_n), (b_1, b_2, \ldots, b_n), (c_1, c_2, \ldots, c_n) \in \mathbb{Q}^n$ and 
$betw(a_i, b_i, c_i)$ for all~$i$, then 
$betw\left(f(a_1, a_2, \ldots, a_n), f(b_1, b_2, \ldots, b_n),f(c_1, c_2, \ldots, c_n)\right)$. Similarly, 
$f\in \mathop{\mathrm{Pol}}(\mathbb{Q},circ)$ if $circ(a_i,b_i,c_i)$ for all~$i$, implies that 
$circ\left(f(a_1, a_2, \ldots, a_n), f(b_1, b_2, \ldots, b_n),f(c_1, c_2, \ldots, c_n)\right)$.

We also study the clone $\mathop{\mathrm{Pol}}(\mathbb{Q},sep)$ of polymorphisms of $(\mathbb{Q}, sep)$, which is the family 
of all $n$-ary functions on~$\mathbb Q$ for $n \ge 0$ that preserve~$sep$. In other words $f : \mathbb{Q}^n \to \mathbb{Q}$ 
lies in $\mathop{\mathrm{Pol}}(\mathbb{Q},sep)$ if 
$(a_1, a_2, \ldots, a_n), (b_1, b_2, \ldots, b_n), (c_1, c_2, \ldots, c_n), (d_1, d_2, \ldots, d_n)\in \mathbb{Q}^n$ and 
$sep(a_i,b_i,c_i,d_i)$ for all~$i$ implies that 
$sep\left(f(a_1, a_2, \ldots, a_n), f(b_1, b_2, \ldots, b_n), f(c_1, c_2, \ldots, c_n), f(d_1, d_2, \ldots, d_n)\right)$.  

There is a corresponding notion of `abstract clone', which we do not require here. Let us note also that the set~
$\mathcal{O}_{A}$ of all finitary operations on~$A$ forms a clone, even a polymorphism clone (e.g., 
$\mathcal{O}_{A} = \mathop{\mathrm{Pol}}(A, =)$). This is the analogue of $\mathop{\mathrm{Sym}}(A)$ for the automorphism 
group and $\mathop{\mathrm{Tr}}(A)$ for the endomorphism monoid. In each of these cases, betweenness, circular, and 
separation relations we write $M$, $E$, and $G$ for the monoids of embeddings, endomorphisms, and the group of automorphisms, 
respectively, and $P$ for the corresponding polymorphism clone. 

The set-up is as follows. An isomorphism $\theta$ is given from $P$ to a closed subclone $P'$ of the full polymorphism clone 
${\cal O}_\Omega$ on a countable set $\Omega$, and our task is to show that it is a homeomorphism. 

Relying on Proposition 27 of~\cite{Bodirsky1}, when proving automatic homeomorphicity of the clone $P$ in each of the cases 
mentioned above, it will suffice to verify that any clone isomorphism between $P$ and a closed clone on some countable set is 
continuous.

\begin{theorem}\label{5.1}
$\mathop{\mathrm{Pol}}(\mathbb{Q}, betw)$ has automatic homeomorphicity, meaning that any isomorphism~$\theta$ from 
$P = \mathop{\mathrm{Pol}}(\mathbb{Q}, betw)$ to a closed subclone~$P'$ of\/~$\mathcal{O}_\Omega$, for a countable
set\/~$\Omega$, is a homeomorphism.
\end{theorem}
\begin{proof} Openness follows from Proposition~27 of~\cite{Bodirsky1}. To demonstrate that~$\theta$ is continuous, we use the 
machinery from section~2 to provide the assumptions of Lemma~5.1 of~\cite{Truss1}. Note that these are properties of the
restriction $\theta\restriction_E\colon P^{(1)}\to P'^{(1)}$, which is a monoid isomorphism between the unary parts 
$P^{(1)} = E$ and $E':= P'^{(1)}$ (these are closed monoids because~$P$ and
$\mathop{\mathrm{Tr}}(\mathbb{Q})$, and~$P'$ and $\mathop{\mathrm{Tr}}(\Omega)$ are closed sets). Namely, we have to verify 
that for every $b\in\Omega$ we can find an endomorphism $h\in E$ with finite image such that 
$\theta(h)(b) = \theta\restriction_E(h)(b)= b$. However, this is precisely the content of Lemma~\ref{2.6} applied
to~$\theta\restriction_E$. \end{proof}

Similarly, using Proposition~27 of~\cite{Bodirsky1}, Lemma~5.1 of~\cite{Truss1} and Lemmas \ref{3.9}, and \ref{4.2}, respectively, one can prove following theorem. 
 
\begin{theorem}\label{5.2} $\mathop{\mathrm{Pol}}(\mathbb{Q}, circ)$, $\mathop{\mathrm{Pol}}(\mathbb{Q}, sep)$ have automatic homeomorphicity.  \end{theorem}

\section{Automatic homeomorphicity of the polymorphism clones generated by monoids}
In this section we show how to `lift' the automatic homeomorphicity results for the polymorphism clones 
$\langle\mathop{\mathrm{End}}\left(\mathbb{Q},<\right)\rangle$ 
and $\langle\mathop{\mathrm{End}}\left(\mathbb{Q},\leq \right)\rangle$ generated by 
$\mathop{\mathrm{End}}\left(\mathbb{Q},<\right)$ and $\mathop{\mathrm{End}}\left(\mathbb{Q},\leq \right)$ respectively
proved in \cite{Truss1} to the reducts discussed earlier in the paper, betweenness, circular order and separation relations.

\begin{theorem} \label{6.3} $\langle\mathop{\mathrm{Emb}}\left(\mathbb{Q},betw \right)\rangle$, $\langle\mathop{\mathrm{Emb}}\left(\mathbb{Q},circ\right)\rangle$, and 
 $\langle\mathop{\mathrm{Emb}}\left(\mathbb{Q},sep \right)\rangle$ have automatic homeomorphicity.
\end{theorem}
\begin{proof}
The following paragraph is an almost verbatim copy of the proof of Lemma 6.5 of \cite{Truss1}. 

We consider the case when $M=\mathop{\mathrm{Emb}}\left(\mathbb{Q},betw \right)$. Let 
$\theta\colon \langle M \rangle \to C$ be a clone isomorphism between $\langle M \rangle$ and another closed clone~$C$ 
on a countable set~$\Omega$. Since by Theorem \ref{2.3}, $M$ has automatic homeomorphicity, and the unary part of~$C$ 
is closed as $C^{(1)} = C \cap \mathop{\mathrm{Tr}}(\Omega)$ and both sets are closed, the restriction 
$\theta \restriction_{M}:M \to C^{(1)}$ is a homeomorphism. By \cite{Truss1} Corollary~6.3 we conclude that~$\theta$ is 
continuous. To see that it must be open too, we use Proposition 32 from \cite{Bodirsky1}, which holds for clone isomorphisms 
and is applicable here since $\mathop{\mathrm{Aut}}\left(\mathbb{Q},betw \right)$ acts transitively on $\mathbb{Q}$ and 
$\theta \restriction_M $ is open. Similarly, we obtain automatic homeomorphicity for $\langle M \rangle$, 
where $M=\mathop{\mathrm{Emb}}\left(\mathbb{Q},circ \right)$, from Theorem \ref{3.7} and Lemma 6.5 of \cite{Truss1}, which is 
applicable here since $\mathop{\mathrm{Aut}}\left(\mathbb{Q},circ \right)$ acts also transitively on $\mathbb{Q}$. 
Finally, we get automatic homeomorphicity for $\langle \mathop{\mathrm{Emb}}\left(\mathbb{Q},sep\right) \rangle$, 
from Theorem \ref{4.1} and Lemma 6.5 of \cite{Truss1}.   \end{proof}

Finally, by appealing to Lemma 6.5 of \cite{Truss1} and Theorems \ref{2.8}, \ref{3.10} and \ref{4.3},  respectively, we get:

\begin{theorem} \label{6.4} $\langle\mathop{\mathrm{End}}\left(\mathbb{Q},betw \right)\rangle$, $\langle\mathop{\mathrm{End}}\left(\mathbb{Q},circ\right)\rangle$, and 
 $\langle\mathop{\mathrm{End}}\left(\mathbb{Q},sep \right)\rangle$ have automatic homeomorphicity.
\end{theorem}

\section{The coloured case}

In this section we remark (without giving full details) how the results earlier in the paper can be easily extended to 
coloured versions. That is, we start with the $C$-coloured version of the rationals ${\mathbb Q}_C$ where 
$2 \le |C| \le \aleph_0$, and form the corresponding reducts, namely the betweenness, circular, and separation relations, and 
obtain analogous automatic homeomorphicity results. Note that, as explained in \cite{Junker}, these are not by any means 
{\em all} the non-trivial reducts, but since exactly what these are is unknown, we just deal with the analogues of the ones 
for $\mathbb Q$. 

We start by considering $({\mathbb Q}_C, \le)$ itself. The main `trick' to deal with this case (and also the circular 
ordering on ${\mathbb Q}_C$) is to define the correct analogue of the classes $\Gamma, \Gamma^+$, $\Gamma^-$, $\Gamma^\pm$ 
(for the circular ordering just $\Gamma$). The main thing to suppose initially is that an injective endomorphism 
$\xi: M \to M$ is given which fixes all members of $G$, and we have to show that it is the identity. Note that exactly as in 
the monochromatic case, $G$ is dense in $M$, so by  Lemma \ref{2.2} automatic homeomorphicity for $M$ follows, since
by \cite{Truss2} Theorem 4.5 we know that ${\rm Aut}({\mathbb Q}_C, \le)$ has the small index property,. 

The definition of $\Gamma$ (etc) in this case is carried out as follows. Given a subset $A$ of ${\mathbb Q}_C$ isomorphic to 
${\mathbb Q}_C$, we again define $x \sim y$ on ${\mathbb Q}_C$ to mean that there is at most one point of $A$ strictly 
between $x$ and $y$. As before, this is an equivalence relation, and the equivalence classes are convex and intersect $A$ in 
at most one point. This time their colours are however relevant. We choose one `extra' colour $c^*$ (i.e. some point not 
lying in $C$), and we colour an equivalence class by the colour of its unique member of $A$, if any, and we colour it by 
$c^*$ otherwise. we can now take $\Gamma$ in this setting to be the family of all $f \in M$ such that ${\mathbb Q}_C$ may be 
written as the disjoint union $\bigcup\{A_q: q \in {\mathbb Q}_{C \cup \{c^*\}}\}$ of convex subsets of ${\mathbb Q}_C$ such 
that $q < r \Rightarrow A_q < A_r$, each $A_q$ is isomorphic to ${\mathbb Q}_C$, if $q \in {\mathbb Q}_{C \cup \{c^*\}}$ is 
coloured by $c \in C$ then $A_q$ has a single point of im $f$, which is coloured $c$, and if $q$ is coloured $c^*$ then $A_q$ 
is disjoint from im $f$. The definitions of $\Gamma^+$, $\Gamma^-$,$\Gamma^\pm$ are similar. 

The lemmas used in \cite{Truss1} to derive automatic homeomorphicity are now transcribed, with appropriate modifications, to 
prove automatic homeomorphicity of ${\rm Emb}({\mathbb Q}_C, \le)$ and ${\rm End}({\mathbb Q}_C, \le)$. 

The passage from the ordered case to the betweenness relation, and from the circular relation to the separation relation are 
performed as before, since the index of the smaller group in the larger is again 2. The main technical lemmas from 
\cite{Truss1} now carry over to the new situation, with colours in $C$ inserted at all the appropriate points, and the 
methods used earlier in this paper used where needed. The conclusion is that the following all have automatic homeomorphicity:
${\rm Emb}({\mathbb Q}_C, \le)$, ${\rm End}({\mathbb Q}_C, \le)$, ${\rm Emb}({\mathbb Q}_C, betw)$, 
${\rm End}({\mathbb Q}_C, betw)$, ${\rm Emb}({\mathbb Q}_C, circ)$, ${\rm End}({\mathbb Q}_C, circ)$,
${\rm Emb}({\mathbb Q}_C, sep)$, ${\rm End}({\mathbb Q}_C, sep)$, the clones generated by all of these, and also 
${\rm Pol}({\mathbb Q}_C, \le)$, ${\rm Pol}({\mathbb Q}_C, betw)$, ${\rm Pol}({\mathbb Q}_C, circ)$, and
${\rm Pol}({\mathbb Q}_C, sep)$.

Briefly, Lemmas \ref{2.1} and \ref{3.8} are readily adapted to the coloured situation, so that for the ordered case, any 
subgroup $H$ of small index may be written as $G_B$ for some uniquely determined finite $B \subset {\mathbb Q}_C$. For the 
reducts, we can still find a unique finite $B$ such that $G_B \le H \le G_{\{B\}}$, and we have the same range of 
possibilities for $H$ as for the monochromatic situation. That is, for the betweenness relation, $H = G_B$ or $G_{\{B\}}$, 
for the circular ordering, $H$ acts on $B$ as a power of some fixed circular map on $B$, and for the separation relation, it 
acts on $B$ as a subgroup of a finite dihedral group. This means that the machinery developed earlier all carries through to 
the coloured case. Note that there are more restrictions here. Thus for instance, for the betweenness relation, the possibility
that $H = G_{\{B\}}$ can only arise if $B$ is `symmetrically' coloured, since otherwise, $B$ will not be preserved setwise by 
any order-reversing automorphism. Similar remarks apply in the other cases. It is still true that all possibilities for $H$ 
must lie in this list, which suffices to make the arguments go through.

\vspace{.1in}

\noindent{\bf Conclusions and problems}

In summary we have given some extensions of the automatic homeomorphicity results of \cite{Bodirsky1} and \cite{Truss1}, by not too complicated modifications of the arguments of the second paper. As 
demonstrated in \cite{Truss1} and here, the methods which apply to the ordered rationals and its reducts have a rather different flavour from those used in \cite{Bodirsky1}. It is not entirely clear how all 
these cases can be extended or generalized. Obvious instances are the (many) remaining reducts of $({\mathbb Q}_C, \le)$ alluded to in the paper. Even to describe what these are may be complicated, as 
explained in \cite{Junker}. A natural extension would be to the case of (2-transitive) trees, originally described in \cite{Droste}, and whose reducts are discussed, in at least one case, in 
\cite{Bodirsky2}, and to more general classes of $\aleph_0$-categorical structures. Finally, we note that our results for the reconstruction of the polymorphism clone apply just to the reflexive case, and 
even for the strict relation on the ordered rationals, this remained open, though an answer has been given in \cite{Barham} (see also the remark at the end of \cite{Bodirsky1}).

\end{document}